\RequirePackage{fix-cm}
\documentclass[smallextended,referee,envcountsect]{svjour3}
\smartqed

\UseRawInputEncoding
\smartqed  
\usepackage{graphicx}
\usepackage{amssymb, graphicx, amsmath}
\usepackage{booktabs}
\usepackage{cases}
\usepackage{subfigure}
\usepackage{color} 

\newcommand{\nn}{\nonumber}
\def \[{\begin{equation}}
\def \]{\end{equation}}
\journalname{JOTA}

\begin{document}

\title{An over-relaxed ADMM for separable convex programming and its applications to statistical learning}

\author{Renyuan Ni}


\institute{Renyuan Ni \at
              Department of Mathematics, Hong Kong Baptist University, Hong Kong, 999077, China \\
              \email{19481535@life.hkbu.edu.hk}           
}

\date{Received: date / Accepted: date}

\maketitle

\begin{abstract}
The alternating direction method of multipliers (ADMM) has been applied successfully in a broad spectrum of areas. Moreover, it was shown in the literature that ADMM  is closely related to the Douglas-Rachford operator-splitting method, and can be viewed as a special case of the proximal point algorithm (PPA). As is well known, the relaxed version of PPA  not only  maintains the convergence properties of PPA in theory, but also numerically outperforms the original PPA. Therefore, it is interesting to ask whether ADMM can be accelerated by using the over-relaxation technique. In this paper, we answer this question affirmatively by developing an over-relaxed ADMM integrated with a criterion to
decide whether a relaxation step is implemented in the iteration.
The global convergence and $O(1/t)$ convergence rate of the over-relaxed ADMM are established. We also implement our proposed algorithm to solve Lasso and sparse inverse covariance selection problems, and compare its performance with the relaxed customized ADMM in \cite{CGHY} and the classical ADMM. 
 The results show that our algorithm substantially outperforms the other two methods.
\end{abstract}


\section{Introduction}
\label{intro}
This paper considers the following structured convex programming problem
\begin{equation}\label{eq:1}
  \min \{\theta_1(x) + \theta_2(y) \; | \;  Ax + By =b, \, x\in {\cal X}, y\in{\cal Y} \},
\end{equation}
where $\theta_1: \Re^{n_1} \to \Re$, $\theta_2:
\Re^{n_2} \to \Re$ are convex functions (but not necessarily smooth), $A\in \Re^{m\times n_1}$, $B\in \Re^{m\times n_2}$ and $b\in
\Re^m$, ${\cal X}\subset \Re^{n_1}$, ${\cal Y}\subset \Re^{n_2}$ are given closed convex sets. {Throughout this paper, the matrices $A$ and $B$ are assumed to have full column-rank.
Problem  \eqref{eq:1} has recently found many applications in a variety of domains such as image processing \cite{Tao2009Alternating,Ng2010Solving}, statistical learning \cite{boyd2011distributed} and communication networking \cite{Combettes2007A,Combettes2009Proximal}.}

The augmented Lagrangian function {of} Problem \eqref{eq:1} is
\begin{equation}\label{eq:2}
{\cal L}_\beta (x,y,\lambda ) = \theta _1(x) + \theta _2(y) - \lambda ^T(Ax + By - b) + \frac{\beta }{2}\left\| {Ax + By - b} \right\|^2,
\end{equation}
where $\lambda$ is the Lagrange multiplier and $\beta > 0$ is a penalty parameter. Notice that Problem \eqref{eq:1} has a separable structure and it is {favorable to take} this advantage in
the algorithmic design.  The alternating direction method of multipliers (ADMM), which can be seen as a splitting version of augmented Lagrangian method, originally  introduced in \cite{Gab76,glowinski1975approximation} is particularly suitable  for solving  \eqref{eq:1}. In fact, the ADMM decomposes the augmented Lagrangian separately with
respect to the variables $x$ and $y$, and then the resulted subproblems can be 
solved individually. At each iteration, the ADMM runs as follows:
\begin{subequations} \label{ADMM}
  \begin{numcases}{\hbox{\quad}}
  x^{k+1} = \arg\min \{{\cal L}_{\beta}(x,y^k,\lambda^k)\,|\, x\in {\cal X}\}, \label{eq:3a}\\[0.2cm]
  y^{k+1} = \arg\min \{{\cal L}_{\beta}(x^{k+1},y, \lambda^k) \,|\, y\in {\cal Y}\},\label{eq:3b}\\[0.2cm]
   \lambda^{k+1} =  \lambda^k- \beta (Ax^{k+1} + By^{k+1} -b).\label{eq:3c}
\end{numcases}
\end{subequations}
Since the ADMM enjoys wide applications, it has been studied extensively in the literature. For example, the ADMM is known to be convergent under mild conditions; see \cite{boyd2011distributed}. Under same conditions, it was shown the ADMM converges with an $O(1/n)$ rate  (where $n$ is the number of iterations)\cite{HeYuan,HeYuan-2}. Furthermore, with some additional assumptions, the works \cite{Boley,Han,Hong,Yang} establish local linear convergence rate results for the ADMM.

Another important issue for the ADMM is to design its accelerated version by slightly modifying the original ADMM with a simple relaxation scheme. Notice that the ADMM was shown to be closely related to the proximal point algorithms(PPA). In fact, ADMM is recovered when a special splitting form of PPA is applied to the dual problem of \eqref{eq:1}. For proximal type algorithms, theirs convergence rates can be improved when an additional over-relaxation step is incorporated on the essential variables; see

Tao proposed a relaxed variant of PPA and proved its linear convergence results\cite{tao2018optimal}. Gu and Yang further proved the optimal linear convergence rate of relaxed PPA under a regularity condition\cite{gu2019optimal}.

As been demonstrated  by Boyd et al. in \cite{boyd2011distributed},  the execution of ADMM is based on the input of $(y^k,\lambda^k)$,  and $x^k$ is not required at all. Thus $x$  plays the role of  intermediary variable  and $(y,\lambda)$ are essential variables in the scheme \eqref{ADMM}.  It is therefore natural to ask whether it is possible to obtain a faster ADMM type method by equipping the ADMM scheme \eqref{ADMM} with a relaxation step on the essential variable $(y^k,\lambda^k)$. This idea leads to the method
\begin{subequations}\label{o-ADMM}
\begin{numcases}{}
x^{k+1}={\rm arg}\!\min_{x\in \mathcal{X}}\{\theta_1(x)-x^TA^T\lambda^k+\frac{\beta}{2}\|Ax+By^k-b\|^2\}, \label{eq:5a}\\
\hat{y}^k={\rm arg}\!\min_{y\in \mathcal{Y}}\{\theta_2(y)-y^TB^T\lambda^{k}+\frac{\beta}{2}\|Ax^{k+1}+By-b\|^2\},\label{eq:5b} \\
\hat{\lambda}^k=\lambda^{k}-\beta(Ax^{k+1}+B\hat{y}^k-b).\label{eq:5c}
\end{numcases}
\end{subequations}
\begin{equation}\label{eq:6}
    \left(\begin{array}{c}
     y^{k+1}\\
    \lambda^{k+1}
   \end{array}
 \right)=
 \left(\begin{array}{c}
     y^{k}\\
    \lambda^{k}
   \end{array}
 \right)-\gamma\left(\begin{array}{c}
     y^{k}-\hat{y}^k\\
    \lambda^{k}-\hat{\lambda}^k
   \end{array}
\right),
\end{equation}
where  the relaxation factor $\gamma\in(0,2)$. Under certain
conditions, researchers have established some results. Some scholars only introduced the relaxation scheme on the essential variable $\lambda^k$. For example, Xu \cite{xu2007proximal} has proved that proximal ADMM can achieve the similar convergence result like the original ADMM when using the relaxation factor $\gamma\in(0,\frac{\sqrt{5}+1}{2})$. Moreover, Ma \cite{ma2020convergence} has derived the feasible region for the relaxation parameter $\gamma$ and proximal parameter and proved the global convergence rate of proximal ADMM. On the other hand, Ye \cite{Ye} proposed a method that modified by using a self-adaptive step size in \eqref{eq:6}, and the performance is observed to be  much improved. Cai \cite{CGHY} proved that if  the order of the $\hat{y}^k$ and $\hat{\lambda}^k$ update is swapped in \eqref{o-ADMM}, the algorithm is convergent. However, it is still not clear whether the method is convergent when we apply the relaxation scheme on the essential variable $(y^k,\lambda^k)$ and $\gamma$ is a constant in $(1,2)$. In this paper, we affirmatively address this question by introducing a simple criterion at each iteration. When the criterion is satisfied, we can over-relax the variables of ADMM with the constant stepsize $\gamma$.  Thus a new relaxed ADMM is proposed and will show that efficiency of the proposed algorithm on optimization problems from statistical learning.

The  paper is organized as follows.  In Section
2, we give a formal statement of the over-relaxed ADMM and provide some preliminaries. Convergence and convergence rate are proved in Section 3. In Section 4, we summarize the experimental results. Finally, some conclusions are made in
Section 5.
%
\section{The Over-relaxed ADMM}
In this section, we describe our over-relaxation ADMM method for solving problem \eqref{eq:1}.
\subsection{Algorithm}
\begin{center}\label{relax-ADMM}
\fbox{
\begin{minipage}[t]{0.9\linewidth}
\smallskip
\begin{subequations}
\textbf{Algorithm: } Over-relaxed ADMM method for solving \eqref{eq:1}.
\begin{numcases}{}
x^{k+1}={\rm arg}\!\min_{x\in \mathcal{X}}\{\theta_1(x)-x^TA^T\lambda^k+\frac{\beta}{2}\|Ax+By^k-b\|^2\},\mbox{ }\mbox{ }\mbox{ }\mbox{ }  \label{eq:8a}\\
\hat{y}^k={\rm arg}\!\min_{y\in \mathcal{Y}}\{\theta_2(y)-y^TB^T\lambda^{k}+\frac{\beta}{2}\|Ax^{k+1}+By-b\|^2\},\mbox{ }\mbox{ }\mbox{ }\mbox{ } \label{eq:8b} \\
\hat{\lambda}^k=\lambda^k-\beta(Ax^{k+1}+B\hat{y}^k-b).\label{eq:8c}
\end{numcases}
\end{subequations}

Criterion: if $(\lambda^{k}-\hat{\lambda}^k)^TB(y^k-\hat{y}^k)\geq 0$, choose $\gamma\in(1,2)$.
\begin{subequations}
\begin{numcases}{}
y^{k+1}=y^k-\gamma(y^k-\hat{y}^k),\label{eq:9b}\\
\lambda^{k+1}=\lambda^k-\gamma(\lambda^k-\hat{\lambda}^k),\label{eq:9c}
\end{numcases}
else
\begin{equation}
y^{k+1}=\hat{y}^k,\quad  \lambda^{k+1}=\hat{\lambda}^k.
\end{equation}
\end{subequations}
\end{minipage}
}
\end{center}
\subsection{Variational characterization of \eqref{eq:1}}
In the following,  we characterize the optimality condition of the problem \eqref{eq:1} as a variational inequality (VI).  The VI reformulation plays a key role in our convergence analysis.

Let the Lagrangian function of \eqref{eq:1} be
\[  \label{Lagrange-2}
 L(x,y,\lambda) =\theta_1(x) + \theta_2(y) - \lambda^T(Ax+By-b).
   \]
Finding primal and dual optimal variables  for \eqref{eq:1} is equivalent to finding a saddle point for the
Lagrangian. Let $(x^*,y^*,\lambda^*)$ be a saddle point of \eqref{Lagrange-2}. Then we have
$$    L_{\lambda\in\Re^m}(x^*,y^*,\lambda) \le L(x^*,y^*,\lambda^*) \le L_{x\in {\cal X}, y\in {\cal Y}}(x,y,\lambda^*). $$
This saddle point problem can be combined into a system
\[  \label{VI-Chara}
    \left\{ \begin{array}{lrl}
     x^*\in {\cal X}, & \theta_1(x) - \theta_1(x^*) + (x-x^*)^T(- A^T\lambda^*) \ge 0,
      & \forall\, x\in {\cal X},
      \\
     y^*\in {\cal Y}, & \theta_2(y) - \theta_2(y^*) + (y-y^*)^T(- B^T\lambda^*) \ge 0,
      & \forall\, y\in {\cal Y},    \\
   \lambda^*\in \Re^m,  &(\lambda-\lambda^*)^T(Ax^*+By^*-b)\ge 0,  &  \forall \; \lambda\in \Re^m.
        \end{array} \right.
       \]
    Compactly, \eqref{VI-Chara} can be written as
       \begin{subequations}\label{VIP-G}
\[  \label{VI-P2-0}
   \hbox{VI}(\Omega,F,\theta)  \qquad      w^*\in \Omega, \quad \theta(u) -\theta(u^*) + (w-w^*)^T F(w^*) \ge 0, \quad \forall \,  w\in
      \Omega,
     \]
where
\[ \label{VI-P2-wF}
   \begin{array}{ccc}
      w = \left(\begin{array}{c}
                     x\\  y\\ \lambda \end{array} \right), \quad
                     u= \left(\begin{array}{c}
                     x\\  y  \end{array} \right),  &&  F(w) = \left(\begin{array}{c}
            - A^T\lambda \\
             -B^T\lambda \\
             Ax + By -b \end{array} \right),  \\[1cm]
     \theta(u)=\theta_1(x) +
                     \theta_2(y),   &\hbox{and}&
            \Omega={\cal X} \times {\cal Y}\times \Re^m.
            \end{array}
   \]
\end{subequations}
In this way,  problem \eqref{eq:1} is reformulated as a variational inequality \eqref{VIP-G}. We denote by $\Omega^*$ the solution set of $\hbox{VI}(\Omega,F,\theta)$. Note that $\Omega^*$ is nonempty under the nonempty assumption of the solution set of \eqref{eq:1}.

\subsection{Preliminaries}
Before proceeding with our analysis,  we introduce some notations which will be  frequently used.
First, for the iterate $(x^{k+1},y^{k+1},\lambda^{k+1})$ generated by the relaxed ADMM (\ref{relax-ADMM}), we define an auxiliary vector $\tilde{w}^k=(\tilde{x}^k,\tilde{y}^k,\tilde{\lambda}^k)$ as
\begin{subequations} \label{Tnotation}
\[  \tilde{x}^k=   x^{k+1}, \qquad   \tilde{y}^k =   \hat {y}^k,
                     \]
and
 \[  \label{TnotationL} \tilde{\lambda}^k =  \lambda^k-\beta(Ax^{k+1}+By^{k}-b).
 \]
 \end{subequations}
Note that $x^{k}$ is an intermediate variable
and $(y,\lambda)$  are essential variables in \eqref{relax-ADMM}, respectively.
Accordingly, for $w =(x,y,\lambda)$ and $w^k =(x^k,y^k,\lambda^k)$ generated by \eqref{relax-ADMM}, we use the notations
$$    v=   \left(\begin{array}{c}
                     y\\  \lambda  \end{array} \right),
                        \quad \hbox{and}   \quad
                     v^k=   \left(\begin{array}{c}
                     y^k\\  \lambda^k  \end{array} \right)
   $$
to denote the essential parts of  $w$ and $w^k$, respectively. We  denote the essential  part of $w^*$ in $\Omega^*$ by use $v^*=(y^*,\lambda^*) $  and let  ${\cal V}^*$ denote all the collection of  $v^*$.

Now we establish the relationship between the iterates $ v^k$ and $v^{k+1}$ generated by the relaxed ADMM \eqref{relax-ADMM}  and the auxiliary variable  defined  by \eqref{Tnotation}.

\begin{lemma} \label{LEM-Correction}For given $v^k=(y^k,\lambda^k)$, let $w^{k+1}$  be generated by the relaxed ADMM \eqref{relax-ADMM} and  ${\tilde w}^k$ be defined  by \eqref{Tnotation}.
Then, we have
\begin{subequations}\label{Correction}
\[ \label{Correction-V}
     v^{k+1} = v^k -  M (v^k-\tilde{v}^k),
      \]
where
\[\label{Matrix-M}
       M  =  \left(\begin{array}{cc}
          \gamma I    &    0 \\
        - \gamma\beta B   & \gamma I_m \end{array} \right).
     \]
\end{subequations}
\end{lemma}
\noindent{Proof.} It follows from \eqref{eq:9c} and \eqref{Tnotation} that
\begin{eqnarray*} 
 \lambda^{k+1}
 & = & \lambda^{k} - \gamma\beta(Ax^{k+1}+B\hat{y}^k-b) \nn \\
  & = & \lambda^k  - \gamma\bigl[\beta(A\tilde{x}^k+By^k-b)-\beta B(y^k-\hat{y}^k)  \bigr] \nn \\
 & = & \lambda^k - \gamma(\lambda^k-\tilde{\lambda}^k) +\gamma\beta B(y^k-\tilde{y}^k).
\end{eqnarray*}
Together with \eqref{eq:9b}, we have the following
  relationship
$$
  \left(\begin{array}{c}
             y^{k+1} \\
              \lambda^{k+1} \end{array} \right)
             =
      \left(\begin{array}{c}
             y^{k} \\
              \lambda^{k} \end{array} \right)
                -    \left(\begin{array}{ccc}
          \gamma I    &    0 \\
              - \gamma\beta B   &\gamma I_m \end{array} \right) \left(\begin{array}{c}
             y^{k}-\tilde{y}^k \\
              \lambda^{k} - \tilde{\lambda}^k \end{array}
              \right).
             $$
The proof is complete. $\hfill$ $\Box$

\section{Convergence analysis}

The following lemma shows the discrepancy of the auxiliary vector $\tilde{w}^k$ from a solution point of VI$(\Omega, F, \theta)$.

\begin{lemma} \label{LEM-Prediction} For given $v^k=(y^k,\lambda^k)$, let $w^{k+1}$  be generated by the relaxed ADMM \eqref{relax-ADMM} and  ${\tilde w}^k$ be defined  by \eqref{Tnotation}.
Then, we have
\begin{subequations}\label{Prediction}
\[ \label{Prediction-W}
   \tilde{w}^k \in \Omega, \;\;
      \theta(u) -\theta(\tilde{u}^k)  + (w- \tilde{w}^k)^T  F(\tilde{w}^k)
          \ge  (v- \tilde{v}^k)^T Q(v^k-\tilde{v}^k), \;\; \forall w\in
          \Omega,
      \]
where
\[  \label{Matrix-Q}
      Q=  \left(\begin{array}{cc}
                \beta B^TB& 0\\
         -B  & \frac{1}{\beta} I_m \end{array} \right). \]
\end{subequations}
\end{lemma}

\noindent{\bf Proof}. The optimality condition of the $x$-subproblem in \eqref {eq:8a} is
$$
  x^{k+1}\in {\cal X}, \;\; \theta_1(x) - \theta_1(x^{k+1})
       + (x-x^{k+1})^T\{-A^T\lambda^k + \beta A^T(Ax^{k+1} + By^k - b) \} \ge 0, \;\; \forall\;  x\in {\cal X}.
       $$
Using the auxiliary vector $\tilde{w}^k$ defined in \eqref{Tnotation}, it can be further written as
\begin{subequations}  \label{ADMM-xlyl}
\[  \label{ADMM-xlyl-x}
    \tilde{x}^{k}\in {\cal X}, \;\;    \theta_1(x) - \theta_1(\tilde{x}^k)
       + (x - \tilde{x}^k)^T(-A^T\tilde{\lambda}^k) \ge 0, \;\; \forall\;  x\in {\cal X}.  \]
Similarly,  the optimality condition of the $y$-subproblem can be written as
        \[
y^{k+1}\in {\cal Y}, \;\;  \theta_2(y)-\theta_2(y^{k+1})+(y-y^{k+1})^T\{-B^T\lambda^k+\beta B^T(Ax^{k+1}+By^{k+1}-b) \}\geq 0,
    \;\; \forall y\in {\cal Y}.
\]
Now, consider the $\{ \cdot\}$ term in the last inequality, we have
        \begin{eqnarray*}
   \lefteqn{y^{k+1}\in {\cal Y}, \;\;  \theta_2(y)-\theta_2(y^{k+1})+(y-y^{k+1})^T\left\{-B^T\big(\lambda^k- (Ax^{k+1} + B y^k-b)\big)\right. } \nn \\
        & &\qquad\qquad \qquad \qquad\qquad\qquad\qquad\qquad\qquad \left.    + \beta B^TB(y^{k+1}-y^k)\right\}\geq 0,
    \;\; \forall y\in {\cal Y}.
\end{eqnarray*}
Using the notations defined in \eqref{Tnotation},  we obtain
\[  \label{ADMM-xlyl-y}
    \tilde{y}^k\in {\cal Y},\;\;  \theta_2(y) - \theta_2(\tilde{y}^k)
      + (y - \tilde{y}^k)^T \bigl\{ -B^T{\tilde{\lambda}}^k + \beta B^TB(\tilde{y}^k - y^k) \bigr\}\ge 0, \;\; \forall y\in {\cal Y}.
             \]
For the definition of $\tilde{\lambda}^k$ given by \eqref{Tnotation},  we have
$$
    (A\tilde{x}^k + B\tilde{y}^k-b) -B(\tilde{y}^k-y^k) + (1/\beta) (\tilde{\lambda}^k-\lambda^k) = 0,
     $$
and it can be written as
\[  \label{ADMM-xlyl-l}
  \tilde{\lambda}^k\in \Re^m, \;\; (\lambda-\tilde{\lambda}^k)^T\bigl\{ (A\tilde{x}^k + B\tilde{y}^k-b) -B(\tilde{y}^k-y^k) + (1/\beta) (\tilde{\lambda}^k-\lambda^k)\bigr\} \ge 0, \;\; \forall\, \lambda\in \Re^m.
    \]
\end{subequations}
Combining \eqref{ADMM-xlyl-x}, \eqref{ADMM-xlyl-y} and \eqref{ADMM-xlyl-l}, and using the notations of \eqref{VIP-G},  the assertion of this lemma is proved. $\hfill$ {$\Box$}

Recall the matrix $M$  defined in (\ref{Matrix-M}) and the matrix $Q$ defined in \eqref{Matrix-Q}, let us define a new matrix $H$ as
\[   \label{H=QM-1}      H=QM^{-1}.\]
The following properties of $H$ will be useful in our analysis.

\begin{proposition} \label{Lemma-PH}
The matrix $H$ defined in (\ref{H=QM-1}) is symmetric and it can be written as
\[ \label{Matrix-H}
       H =  \frac{1}{\gamma}\left(\begin{array}{cc}
   \beta B^TB & 0 \\
    0 & \frac{1}{\beta}I
   \end{array}
 \right).
\]
Moreover, when $B$ is full column rank and  $\gamma\in (0,2)$ and $\mu \ge \alpha$,  $H$
 is positive definite.
\end{proposition}
\begin{proof}
  The proof is simple, so we omit it here.
\end{proof}

Lemma \ref{LEM-Prediction} shows that the accuracy of the generated iterate to the solution
point of   $\hbox{VI}(\Omega,F,\theta)$ is measured by the term $(v- \tilde{v}^k)^T Q(v^k-\tilde{v}^k)$. Therefore,
we need to further investigate the term $(v- \tilde{v}^k)^T Q(v^k-\tilde{v}^k)$.  Our purpose is to  write it  in terms of $\|v-v^k\|_H$ and $\|v-v^{k+1}\|_H$. This is given in the next lemma.

\begin{lemma}
For given $v^k=(y^k,\lambda^k)$, let $w^{k+1}$  be generated by the relaxed ADMM \eqref{relax-ADMM} and  ${\tilde w}^k$ be defined  by \eqref{Tnotation}.
Then, we have
  \begin{equation}\label{lemma-2}
  (v-\tilde{v}^k)^TH(v^k-v^{k+1})=\frac{1}{2}(\|v-v^{k+1}\|_H^2-\|v-uv^k\|_H^2)+\frac{1}{2}\|v^k-\tilde{v}^k\|_G^2,\ \forall v\in\Omega,
  \end{equation}
  where the matrix $G=Q^T+Q-M^THM$.
\end{lemma}
\begin{proof}
Applying the identity
    \begin{equation}
    (a-b)^TH(c-d)=\frac{1}{2}(\|a-d\|_H^2-\|a-c\|_H^2)+\frac{1}{2}(\|c-b\|_H^2-\|d-b\|_H^2),
  \end{equation}
  to the right term of (\ref{Prediction-W}) with $a=v, b=\tilde{v}^k, c=v^k, d=v^{k+1}$, we obtain
   \begin{equation}\label{Q-equation-2}
      \begin{split}
  (v-\tilde{v}^k)^TH(v^k-v^{k+1})&=\frac{1}{2}(\|v-v^{k+1}\|_H^2-\|u-v^k\|_H^2)\nn\\
 &\qquad +\frac{1}{2}(\|v^k-\tilde{v}^k\|_H^2-\|v^{k+1}-\tilde{v}^k\|_H^2).
   \end{split}
\end{equation}
For the last term of (\ref{Q-equation-2}), we have
   \begin{equation}\label{Q-equation-3}
   \begin{split}
    \|v^k&-\tilde{v}^k\|_H^2-\|v^{k+1}-\tilde{v}^k\|_H^2\\
    &=\|v^k-\tilde{v}^k\|_H^2-\|(v^k-\tilde{v}^k)-(v^k-v^{k+1})\|_H^2\\
     &\stackrel{(\ref{Correction})}{=}\|v^k-\tilde{v}^k\|_H^2-\|(v^k-\tilde{v}^k)-M(v^k-\tilde{v}^k)\|_H^2\\
  & =2(v^k-\tilde{v}^k)^THM(v^k-\tilde{v}^k)-(v^k-\tilde{v}^k)^TM^THM( v^k-\tilde{v}^k)\\
    &\stackrel{(\ref{Matrix-H})}{=}(v^k-\tilde{v}^k)^T(Q^T+Q-M^THM)(v^k-\tilde{v}^k).
   \end{split}
\end{equation}
the assertion is proved.
\end{proof}

Now, we investigate the properties of the matrix G. Using \eqref{Matrix-H}, we have
\begin{eqnarray} \label{Matrix-G-P}
       G &= &(Q^T+ Q)- M ^TQ    \nn \\
        & =&  \left(\begin{array}{cc}
                2\beta B^TB& -B^T\\
         -B  & \frac{2}{\beta} I_m \end{array} \right)
          -   \left(\begin{array}{cc}
          \gamma I    &    - \gamma\beta B^T \\
           0 & \gamma I_m \end{array} \right)\left(\begin{array}{cc}
                \beta B^TB& 0\\
         -B  & \frac{1}{\beta} I_m \end{array} \right)  \nn \\
         & = &  \left(\begin{array}{cc}
                2\beta B^TB& -B^T\\
         -B  & \frac{2}{\beta} I_m \end{array} \right)- \left(\begin{array}{cc}
                2\gamma \beta B^TB& -\gamma B^T\\
         -\gamma B  & \frac{\gamma }{\beta} I_m \end{array} \right)\nn \\
         & = &\left(\begin{array}{cc}
    (2-2\gamma)\beta B^TB\; & \;(\gamma-1) B^T \\
    (\gamma-1) B & \frac{2-\gamma}{\beta} I
   \end{array}
 \right).
    \end{eqnarray}

\vspace{4cm}

\begin{lemma}\label{lem1}
Let {${v^k}$} be the sequence generated by a method for the problem \eqref{eq:1} and ${\tilde w ^{\rm{k}}}$ is obtained in the $k$-th iteration. If H defined in the Theorem \ref{thm1} is positive definite,then we have
\begin{equation}\label{eq:15}
\left\| {{v^{k + 1}} - {v^ * }} \right\|_H^2 \le \left\| {{v^k} - {v^ * }} \right\|_H^2 - \left\| {{v^k} - {{\tilde v}^k}} \right\|_G^2,\quad \forall {v^ * } \in {V^ * }.
\end{equation}
\end{lemma}
In the following, we further investigate the term $\left\| {{v^k} - {{\tilde v}^k}} \right\|_G^2$ and show how to bound it.\\
Because $G = \left( {\begin{array}{*{20}{c}}
{(2 - 2r)\beta {B^T}B}&{(r - 1){B^T}}\\
{(r - 1)B}&{\frac{{2 - r}}{\beta }I}
\end{array}} \right)\quad and\quad v = \left( \begin{array}{l}
y\\
\lambda
\end{array} \right)$,we have
 \begin{equation}\label{eq:23}
    \begin{split}
   \|v^k-\tilde{v}^k\|_G^2&= (2-2r)\beta\|B(y^k-\tilde{y}^k)\|^2+\frac{2-r}{\beta}\|\lambda^k-\tilde{\lambda}^k\|^2\\
   &+2(r-1)(\lambda^{k}-\tilde{\lambda}^{k})^TB(y^k-\tilde{y}^k)\\
   &=(2-2r)\beta\|B(y^k-\tilde{y}^k)\|^2+\frac{2-r}{\beta}\|\beta(A\tilde{x}^k+By^k-b)\|^2\\
   &\quad+2(r-1)\beta(A\tilde{x}^k+By^k-b)^TB(y^k-\tilde{y}^k)\\
      &=(2-2r)\beta\|B(y^k-\tilde{y}^k)\|^2+\frac{2-r}{\beta}\|\beta(A\tilde{x}^k+B\tilde{y}^k-b)+\beta B(y^k-\tilde{y}^k)\|^2\\
      &\quad+2(r-1)\{\beta(A\tilde{x}^k+B\tilde{y}^k-b)+\beta B(y^k-\tilde{y}^k)\}^TB(y^k-\tilde{y}^k)\\
            &=(2-2r)\beta\|B(y^k-\tilde{y}^k)\|^2+(2-r)\beta\|B(y^k-\tilde{y}^k)\|^2\\
           &\quad+\frac{2-r}{\beta}\|\beta(A\tilde{x}^k+B\tilde{y}^k-b)\|^2+(4-2r)\beta(A\tilde{x}^k+B\tilde{y}^k-b)^T B(y^k-\tilde{y}^k)\\
      &\quad+2(r-1)\beta(A\tilde{x}^k+B\tilde{y}^k-b)^TB(y^k-\tilde{y}^k)+2(r-1)\beta \|B(y^k-\tilde{y}^k)\|^2\\
   &=(2-r)\beta\|B(y^k-\tilde{y}^k)\|^2+\frac{2-r}{\beta}\|\beta(A\tilde{x}^k+B\tilde{y}^k-b)\|^2\\
   &\quad+2\beta(A\tilde{x}^k+B\tilde{y}^k-b)^TB(y^k-\tilde{y}^k)
    \end{split}
   \end{equation}
Notice that
\begin{equation*}
\hat{\lambda}^k=\lambda^{k}-\beta(A\tilde{x}^k+B\tilde{y}^k-b)
\end{equation*}
and
\begin{equation*}
y^k-\tilde{y}^k=\frac{1}{r}(y^k-y^{k+1})
\end{equation*}
Thus, we have
\begin{equation}\label{eq:24}
    \begin{split}
   \|v^k-\tilde{v}^k\|_G^2
   &=\frac{2-r}{r^2}\beta\|B(y^k-y^{k+1})\|^2+\frac{2-r}{\beta}\|\lambda^{k}-\hat{\lambda}^k\|^2\\
   &\quad+2(\lambda^{k}-\hat{\lambda}^k)^TB(y^k-\tilde{y}^k)
    \end{split}
\end{equation}
Meanwhile, we have
\begin{equation*}
\lambda^{k}-\hat{\lambda}^k=\frac{1}{r}(\lambda^k-\lambda^{k+1})
\end{equation*}
Substituting it in \eqref{eq:24}, we obtain
 \begin{equation}\label{eq:25}
    \begin{split}
   \|v^k-\tilde{v}^k\|_G^2
   &=\frac{2-r}{r^2}\beta\|B(y^k-y^{k+1})\|^2+\frac{2-r}{r^2\beta}\|\lambda^k-\lambda^{k+1}\|^2\\
   &\quad+2(\lambda^{k}-\hat{\lambda}^k)^TB(y^k-\tilde{y}^k)
    \end{split}
   \end{equation}
 Combining \eqref{eq:23} and \eqref{eq:25},we have
  \begin{equation}\label{eq:26}
      \begin{split}
   \|v^{k+1}-v^*\|_H^2&\leq\|v^{k}-v^*\|_H^2-\frac{2-r}{r^2}\beta\|B(y^k-y^{k+1})\|^2\nn\\
  &\qquad -\frac{2-r}{r^2\beta}\|\lambda^k-\lambda^{k+1}\|^2-2(\lambda^{k}-\hat{\lambda}^k)^TB(y^k-\tilde{y}^k).
       \end{split}
   \end{equation}
Now, we give the following theorem to establish the global convergence of our algorithm for $1<r<2$.
\begin{theorem}
For the sequence generated {${w ^k}$} by the over relaxation ADMM, if $(\lambda^{k}-\hat{\lambda}^k)^TB(y^k-\tilde{y}^k)\geq 0$, we have
\begin{equation}\label{eq:32}
\mathop {\lim }\limits_{k \to \infty } ({\left\| {B({y^k} - {y^{k + 1}})} \right\|^2} + {\left\| {{\lambda ^k} - {\lambda ^{k + 1}}} \right\|^2}) = 0.
\end{equation}
and the sequence {${v^k}$} converges
to a solution point ${v^\infty } \in {{\cal V}^ * }$.
\end{theorem}
\begin{proof}
Let $({y^0},{\lambda ^0})$ be the initial iterate. For $1<r<2$, according to \eqref{eq:26}, there are constants ${C_1} = \frac{{2 - r}}{{{r^2}}}\beta  > 0,\;and\;{C_2} = \frac{{2 - r}}{{{r^2}\beta }} > 0$ such that
\begin{equation}\label{eq:33}
{C_1}{\left\| {B({y^k} - {y^{k + 1}})} \right\|^2} + {C_2}{\left\| {{\lambda ^k} - {\lambda ^{k + 1}}} \right\|^2} \le \left\| {{v^k} - {v^*}} \right\|_H^2 - \left\| {{v^{k + 1}} - {v^*}} \right\|_H^2  \quad \forall {v^*} \in {{\cal V}^* }
\end{equation}
Summing the above inequality from $k=0$ to $\infty$, we obtain
\begin{equation*}
\sum\limits_{k = 0}^\infty  {({C_1}{{\left\| {B({y^k} - {y^{k + 1}})} \right\|}^2} + {C_2}{{\left\| {{\lambda ^k} - {\lambda ^{k + 1}}} \right\|}^2})}  \le \left\| {{v^0} - {v^*}} \right\|_H^2
\end{equation*}
and thus we obtain the assertion \eqref{eq:32}.\\
Furthermore, it follows from \eqref{eq:33} that the sequence {${v^k}$} is in a compact set and it has a subsequence ${{v}^{{k_j}}}$ converging to a cluster point, say, ${v^\infty }$. Let ${{\tilde x}^\infty }$ be generated by $({y^\infty },{\lambda ^\infty })$. Because $B$ assumed to be full column rank and equation \eqref{eq:32}, we have
\begin{equation*}
B({y^\infty } - {{\tilde y}^\infty }) = 0\quad and\quad {\lambda ^\infty } - {{\tilde \lambda }^\infty } = 0
\end{equation*}
Then, according to \eqref{Prediction-W}, we get
\begin{equation*}
\theta (u) - \theta ({{\tilde u}^\infty }) + {(w  - {{\tilde w }^\infty })^T}F({{\tilde w }^\infty }) \ge 0,\quad\forall w  \in \Omega ,
\end{equation*}
So ${{\tilde w }^\infty } = {w ^\infty }$ is a solution point of \eqref{VI-P2-0}. Also, \eqref{eq:33} is true for any solution point of \eqref{VI-P2-0}, we have
\begin{equation*}
\left\| {{v^{k + 1}} - {v^\infty }} \right\|_H^2 \le \left\| {{v^k} - {v^\infty }} \right\|_H^2
\end{equation*}
the sequence {${v^k}$} cannot have another cluster point and thus it converges to a solution point ${v^ * } = {v^\infty } \in {{\cal V}^ * }$.
\end{proof}
Because $y^k$ is the corrector of $\tilde{y}^{k-1}$, $\lambda^{k}$ is the corrector of $\hat{\lambda}^{k-1}$. In most cases, $(\lambda^{k}-\hat{\lambda}^k)^TB(y^k-\tilde{y}^k)\geq0$ is true in practical computation, while only a very small part of cases do not hold. The frequency of not being true depends on whether $r$ is close to 2. Therefore, we can first judge whether $(\lambda^{k}-\hat{\lambda}^k)^TB(y^k-\tilde{y}^k)\geq 0$  is true or not. If it is true, we can use over relaxation method to set $r$ from 0 to 2; otherwise we let $r=1$, which is the original ADMM method.
\section{Numerical Experiment}


 In this paper, we report the numerical performance of our over-relax ADMM  by solving the Lasso problem and the Sparse inverse covariance selection. We compare our over-relax ADMM  with the relaxed customized PPA proposed in \cite{CGHY}. All the simulations are performed on a Laptop with 8GB RAM memory, using Matlab R2015b.

\subsection{Lasso Problem}
The least absolute shrinkage and selection operator (Lasso) was first prososed by Tibshirani in his famous work \cite{tibshirani1996regression}, and has been used as a very popular and attractive method for regularization and variable selection for high-dimensional data in statics and machine learning.  Lasso can be mathematically formulated as
\begin{equation}\label{lasso}
  \mbox{min}\ \frac{1}{2}\|Ax-b\|_2^2+\rho\|x\|_1,
\end{equation}
where $\rho>0$ is a  regularization parameter. The $m\times n$ dimensional matrix $A$ contains the $m$-dimensional independent observations of $x$, and the vector $b$ contains the $m$-dimensional independent observations of the response variable.  The main reason for the popularity of Lasso is that it can handle the case where predictor variables are larger than samples and produce sparse models which are easy to interpret.

Now, we show how to use the  ADMM   for solving Lasso (\ref{lasso}). First, we rewrite the original problem (\ref{lasso}) into the canonical form in form of  (\ref{eq:1}) that consists of two sparable variables. Introducing a new variable $y=x$, the problem can be written as
\begin{equation}\label{eq:lasso}
    \begin{split}
\mbox{min}\quad & \frac{1}{2}\|Ax-b\|_2^2+\rho\|y\|_1\\
\mbox{s.t.}\quad &x-y=0.
\end{split}
\end{equation}
Then apply ADMM  to (\ref{eq:lasso}), we get
\begin{subequations}
\begin{numcases}{}
x^{k+1}={\rm arg}\!\min\{\frac{1}{2}\|Ax-b\|_2^2+\frac{\beta}{2}\|x-y^k-\frac{1}{\beta}z^k\|_2^2\},\label{ADMcs-la-1}\\
y^{k+1}={\rm arg}\!\min\{\rho\|y\|_1+\frac{\beta}{2}\|x^{k+1}-y-\frac{1}{\beta}z^k\|_2^2\}.\label{ADMcs-la-2}\\
z^{k+1}=z^{k}-\beta\left(x^{k+1}-y^{k+1}\right),\label{ADMcs-la-3}
\end{numcases}
\end{subequations}

The first subproblem is an unconstrained  convex quadratic minimization.
Its unique solution can be given by
\begin{equation}\label{ADMcs-sov1}
x^{k+1}:=(A^TA+\beta I)^{-1}(A^Tb+\beta y^k+z^k ).
\end{equation}

Applying the Sherman$-$Morrison formula to \eqref{ADMcs-sov1}, we can obtain
\begin{equation}
x^{k+1}:=\left[\frac{1}{\beta}I-\frac{1}{\beta^3}A^T(\beta I+AA^T)^{-1}A\right](A^Tb+\beta y^k+z^k).
\end{equation}

The second subproblem is a $l_1+l_2$ minimization,  and it can be solved via using  the soft thresholding operator
\begin{equation}\label{ADMcs-sov2}
\tilde{y}^k:=\mathcal{S}_{\frac{\rho}{\beta}}(x^{k+1}+\tilde{z}^k),
\end{equation}
where the operator is defined by
\begin{equation}\label{soft-oper}
\mathcal{S}_{\kappa}(a)=(a-\kappa)_{+}-(-a-\kappa)_+.
\end{equation}
We refer the reader to \cite{boyd2011distributed} for more details.

The data was generated similarly with   \cite{boyd2011distributed}, and as follows.  For different size of $m $ and $n$, we generate the matrix $A$  with independent standard Gaussian
values, then we standardize the columns of $A$ to have unit $l_2$ norm. A `true' value $x_{\textrm{true}}\in\Re^n$ is generated with around 100 nonzero entries, each sampled from an $\mathcal{N}(0,1)$ distribution. The labels $b$ are then computed as $b=Ax_{\textrm{true}}\in\Re^n+v$, where $v\sim \mathcal{N}(0,10^{-3}I)$, which corresponds to a signal-to-noise ratio $\|Ax_{\textrm{true}}\|_2^2/\|v\|_2^2$ of around 200.

The regularization parameter was determined by $\rho=0.1\rho_{\textrm{max}}$, where $\rho_{\textrm{max}}=\|A^Tb\|_\infty$ is the critical value of $\rho$ above which the solution of the lasso problem is $x=0$. The proximal parameter  for the ADMM and relaxed ADMM is set to be $\beta=\frac{1}{\delta}$. The  termination tolerance is defined as similarly with   \cite{boyd2011distributed}, that is,
\begin{equation}
\|r^{k}\|_2=\|x^{k}-y^{k}\|_2\le \varepsilon^{\mbox{pri}},\qquad \|s^{k}\|_2=\|y^{k}-y^{k-1}\|_2\le \varepsilon^{\mbox{dual}}.
\end{equation}
where
\begin{equation}\label{error}
    \begin{split}
&\varepsilon^{\mbox{pri}}=\sqrt{p}\varepsilon^{\mbox{abs}} +\varepsilon^{\mbox{rel}}\max\{\|x^k\|_2,\|y^k\|_2\}\\
&\varepsilon^{\mbox{dual}}=\sqrt{n}\varepsilon^{\mbox{abs}} +\varepsilon^{\mbox{rel}} \|y^k\|_2
    \end{split}
\end{equation}
and  we set  $\varepsilon^{\mbox{abs}}=10^{-5},\ \varepsilon^{\mbox{rel}}=10^{-3}$ for both methods. All variables were initially set to be zero. Since $(m<n)$ in matrix $A$, we use the Sherman$–$Morrison formula to $(A^TA+\beta I)^{-1}$ and instead factor the smaller matrix $I+\beta AA^T$, which is then cached for subsequent $x$-updates.

Table 1-3 gives the iterations required and the total computational time for ADMM, relaxed customized ADMM and our relaxed ADMM with $\gamma=1.8$. From the table, we can see that, the relaxed ADMM is always faster than the other two algorithms and it requires less iteration steps and less time to reach the termination tolerance. Thus, relaxed ADMM is more efficient than the original ADMM and the relaxed customized ADMM. The numerical results of three methods for the case with $m=1000,n=1500$ are plotted in Fig 1. The left part in Fig 1 presents the relationship of primal residual $r^k$ and the number of iteration while the right part shows the relationship of dual residual $s^k$ and the number of iteration. The green curve plots our relaxed ADMM and the other two curves belongs to ADMM and relaxed customized ADMM respectively. We can observe our method converges linearly and is considerably faster than other two methods, which supports our convergence analysis.

\begin{table*}[!ht]
\newcommand{\cell}[1]{#1}
\newcommand{\cellc}[1]{\textrm{\color{blue}#1}}
\begin{center}
\caption{Performance comparison of ADMM, Relaxed ADMM and Relaxed customized ADMM($\varepsilon^{\mbox{abs}}=10^{-5},\ \varepsilon^{\mbox{rel}}=10^{-3}$) }\label{pca1}
\vskip0.25cm
\begin{tabular}{|cc|c|c|c|c|c|c|c|c|c|c|c|c|}\hline
\multicolumn{2}{|c|}{$ m\times n$  matrix} & \multicolumn{4}{|c|}{ ADMM}& \multicolumn{4}{|c|}{ Relaxed ADMM} & \multicolumn{4}{|c|}{Relaxed customized ADMM}  \\ \hline
 $m$ & $n $    &  \cell{ Iter.  }&  \cell{$\|r^{k}\|_2$  }&  \cell{$\|s^{k}\|_2$   } &  \cell{Time }  & \cell{Iter.   } &  \cell{$\|r^{k}\|_2$   }&  \cell{$\|s^{k}\|_2$ }&  \cell{Time } &  \cell{ Iter.  }&  \cell{$\|r^{k}\|_2$  }&  \cell{$\|s^{k}\|_2$   } &  \cell{Time } \\ \hline
1000	 & 1500	 &   18	 &  5.54e-03 &  1.45e-03 &  0.10	&   14&  5.94e-03&  3.54e-03 & 0.09 &  28 & 2.04e-03& 3.43e-03	&  0.14 \\ \hline
1500	 & 1500	 &   16	 &  5.16e-03 &  2.26e-03 &  0.16	&  22&  6.09e-03&  1.06e-03 &  0.17 &  30	 & 1.47e-03	 & 3.26e-03&  0.19 \\ \hline
1500	 & 3000	 &  20	 &  9.23e-03 &  2.42e-03 & 0.38	&  20& 9.23e-03&  2.42e-03 & 0.38 &  27 & 4.01e-03 & 5.71e-03&  0.41 \\	 \hline
2000	 & 3000	 &  18	 &  7.70e-03 &  3.63e-03 &  0.55 &  15&  5.30e-03&  4.33e-03 & 0.50&  29 & 2.83e-03	& 4.72e-03&  0.72\\	 \hline
3000	 & 3000	 &   16	 &  6.97e-03 &  2.88e-03 &  0.59	&   13&  6.81e-03&  2.43e-03 & 0.52&   28	 & 2.29e-03& 5.05e-03	 & 0.79\\	 \hline
3000	 & 5000	 &   18	 &  1.26e-02 &  4.17e-03 &  1.48&   17	&  8.59e-03&  6.06e-03 & 1.39&   27 & 4.78e-03 & 7.30e-03 &  1.75 \\	\hline
4000	 & 5000	 &   17	 &  9.99e-03 &  3.53e-03 &  2.23	&  14&  1.05e-02&  6.76e-03 & 2.14	&  28& 3.04e-03 & 5.67e-03 & 2.97  \\ \hline	
5000	 & 5000	 &  16	 &  1.00e-02 &  4.69e-03 &  1.76	&   15&  9.99e-03&  4.53e-03 & 1.64&   28	& 2.99e-03	& 6.64e-03	& 2.25 \\	 \hline	
5000	 & 10000	 &   20	 &  1.70e-02 &  6.59e-03 &   5.40	 &   20&  1.70e-02&  6.59e-03 &  5.29&   26& 7.72e-03& 1.06e-02& 6.03 \\ \hline	
7000	 & 10000	 &   18	 &  1.57e-02 &  3.78e-03 &   9.37	 &   16&  1.02e-02&  5.64e-03 &  8.98&   23& 6.85e-03& 1.11e-02& 10.53 \\	 \hline	
10000	 & 10000	 &   16	 &  1.19e-02 &  6.15e-03 &    7.76	 &   12&  1.71e-02&  9.33e-03 &   7.31&   29& 4.07e-03& 8.96e-03& 10.14\\ \hline	
\end{tabular}
\end{center}
\vskip-0.4cm
\end{table*}

\begin{table*}[!ht]
\newcommand{\cell}[1]{#1}
\newcommand{\cellc}[1]{\textrm{\color{blue}#1}}
\begin{center}
\caption{Performance comparison of ADMM, Relaxed ADMM and Relaxed customized ADMM($\varepsilon^{\mbox{abs}}=10^{-6},\ \varepsilon^{\mbox{rel}}=10^{-4}$) }\label{pca2}
\vskip0.25cm
\begin{tabular}{|cc|c|c|c|c|c|c|c|c|c|c|c|c|}\hline
\multicolumn{2}{|c|}{$ m\times n$  matrix} & \multicolumn{4}{|c|}{ ADMM}& \multicolumn{4}{|c|}{ Relaxed ADMM} & \multicolumn{4}{|c|}{ Relaxed customized ADMM}  \\ \hline
 $m$ & $n $    &  \cell{ Iter.  }&  \cell{$\|r^{k}\|_2$  }&  \cell{$\|s^{k}\|_2$   } &  \cell{Time }  & \cell{Iter.   } &  \cell{$\|r^{k}\|_2$   }&  \cell{$\|s^{k}\|_2$ }&  \cell{Time }  &  \cell{ Iter.  }&  \cell{$\|r^{k}\|_2$  }&  \cell{$\|s^{k}\|_2$   } &  \cell{Time } \\ \hline
1000	 & 1500	 &  27	 &  7.04e-04 &  7.19e-05 &  0.17&   24	&  7.09e-04&  8.16e-05 &  0.15 &   36	 & 2.10e-04	 & 3.43e-04&  0.19\\  \hline
1500	 & 1500	 &  23	 &  6.28e-04 &  5.94e-05 &  0.18&   19&  4.46e-04&  5.93e-05 &  0.16 &   38 & 1.66e-04	 & 3.69e-04&  0.26 \\ \hline
1500	 & 3000	 &  31	 &  9.51e-04 &  1.13e-04  &   0.47&   25&  8.32e-04&  1.84e-04 &  0.43&   36 & 4.36e-04 & 5.97e-04&  0.55 \\ \hline
2000	 & 3000	 &   28	 &  9.04e-04 &  9.02e-05 &  0.66&   22&  7.37e-04&  1.17e-04 &  0.58&   38 & 3.39e-04 & 5.59e-04&  0.84 \\ \hline
3000	 & 3000	 &   24	 &  8.12e-04 &  7.15e-05 &   0.72	&   22&  6.03e-04&  5.62e-05 &  0.69&   38& 2.12e-04	 & 4.72e-04	&  0.95 \\ \hline
3000	 & 5000	 &   30	 &  8.85e-04 &  8.12e-05 &   1.81	 &   24&  9.09e-04&  1.28e-04 &  1.63	&   34& 5.00e-04& 7.34e-04&  1.98 \\ \hline
4000	 & 5000	 &   26	 &  9.85e-04 &  1.28e-04 &   2.77	 &   21&  9.77e-04&  2.61e-04 &  2.49&   36& 3.26e-04& 6.00e-04&  3.29 \\ \hline	
5000	 & 5000	 &   24	 &  1.02e-03 &  1.07e-04 &    2.07 &   18&  9.55e-04&  1.40e-04 &  1.74	&   38& 2.80e-04& 6.23e-04& 2.70 \\ \hline 5000	 & 10000	 &  32	 &  1.67e-03 &  1.81e-04 &     6.62 &   26	&  1.75e-03&  3.09e-04 & 5.96	 &   35	& 7.84e-04& 1.06e-03& 6.94\\ \hline	7000	 & 10000	 &   28	 &  1.42e-03 &  1.50e-04 &     11.35&   25&  1.14e-03&  1.34e-04 &  10.58&   34	 & 6.78e-04&1.11e-03& 11.99\\ \hline	10000	 & 10000	 &   23	 &  1.73e-03 &  1.81e-04 &     9.03&  19&  1.65e-03&  4.37e-04 &  8.76&   38	 & 4.24e-04& 9.34e-04& 11.84\\	 \hline	
\end{tabular}
\end{center}
\vskip-0.4cm
\end{table*}

\begin{table*}[!ht]
\newcommand{\cell}[1]{#1}
\newcommand{\cellc}[1]{\textrm{\color{blue}#1}}
\begin{center}
\caption{Performance comparison of ADMM, Relaxed ADMM and Relaxed customized ADMM($\varepsilon^{\mbox{abs}}=10^{-7},\ \varepsilon^{\mbox{rel}}=10^{-5}$) }\label{pca3}
\vskip0.25cm
\begin{tabular}{|cc|c|c|c|c|c|c|c|c|c|c|c|c|}\hline
\multicolumn{2}{|c|}{$ m\times n$  matrix} & \multicolumn{4}{|c|}{ ADMM}& \multicolumn{4}{|c|}{ Relaxed ADMM} & \multicolumn{4}{|c|}{ Relaxed customized ADMM}  \\ \hline
 $m$ & $n $    &  \cell{ Iter.  }&  \cell{$\|r^{k}\|_2$  }&  \cell{$\|s^{k}\|_2$   } &  \cell{Time }  & \cell{Iter.   } &  \cell{$\|r^{k}\|_2$   }&  \cell{$\|s^{k}\|_2$ }&  \cell{Time } & \cell{Iter.   } &  \cell{$\|r^{k}\|_2$   }&  \cell{$\|s^{k}\|_2$ }&  \cell{Time } \\ \hline
1000	 & 1500	 &   38	 &  6.51e-05 &  5.42e-06 & 0.20&   31&  6.63e-05&  5.20e-06 &  0.17 & 49 & 2.16e-05	& 3.62e-05	 &  0.25 \\  \hline
1500	 & 1500	 &   33	 &  6.37e-05 &  5.48e-06 & 0.22&   24&  5.94e-05&  8.77e-06 &  0.18&   48 & 1.82e-05& 4.04e-05&  0.28 \\ \hline
1500	 & 3000	 &   43	 &  9.60e-05 &  7.08e-06 & 0.62&   35&  1.03e-04&  1.55e-05 &   0.50&   50	 & 3.34e-05	 & 4.88e-05& 0.66\\	 \hline
2000	 & 3000	 &   40	 &  7.96e-05 &  4.83e-06 &  0.90	&   31&  8.21e-05&  4.87e-06 &  0.75&   44 & 3.69e-05	 & 5.81e-05& 0.93 \\	 \hline
3000	 & 3000	 &   33	 &  8.98e-05 &  7.18e-06 &  0.88	 &   27&  8.56e-05&  6.74e-06 &  0.75	&  49 & 2.15e-05 & 4.76e-05	 & 1.13 \\	 \hline
3000	 & 5000	 &   40	 &  1.09e-04 &  9.12e-06 &   2.32	 &   32	&  9.37e-05&  8.12e-06 &  1.96&   46& 4.85e-05& 7.38e-05	 &2.53 \\	 \hline
4000	 & 5000	 &   34	 &  1.27e-04 &  1.18e-05 &    3.27	 &   27	&  1.24e-04&  1.19e-05 &   2.84&   51& 2.98e-05& 5.76e-05	 &4.13 \\	 \hline	
5000	 & 5000	 &   33	 &  1.07e-04 &  8.72e-06 &   2.62	 &   26	&  1.01e-04&  7.79e-06 &   2.21&   49& 2.91e-05 & 6.43e-05&3.39 \\	 \hline	
5000	 & 10000	 &   46	 &  1.35e-04 &  8.40e-06 &   8.48 &   36	&  1.23e-04&  1.43e-05 &   7.20&   41& 1.03e-04& 1.37e-04	 & 7.69\\	 \hline	
7000	 & 10000	 &   38	 &  1.61e-04 &  1.14e-05 &     13.16&   31	&  1.62e-04&  1.13e-05 &   12.30&  45& 6.61e-05& 1.08e-04 &14.54\\ \hline	10000	 & 10000	 &   34	 &  1.22e-04 &  9.52e-06 & 11.99&   25 &  1.12e-04&  8.56e-06 &   10.32&  46& 4.47e-05& 9.98e-05 &14.33 \\ \hline	
\end{tabular}
\end{center}
\vskip-0.4cm
\end{table*}

\begin{figure*}[ht!]
         \begin{center}
  \centerline{
  \scalebox{0.4}{\includegraphics{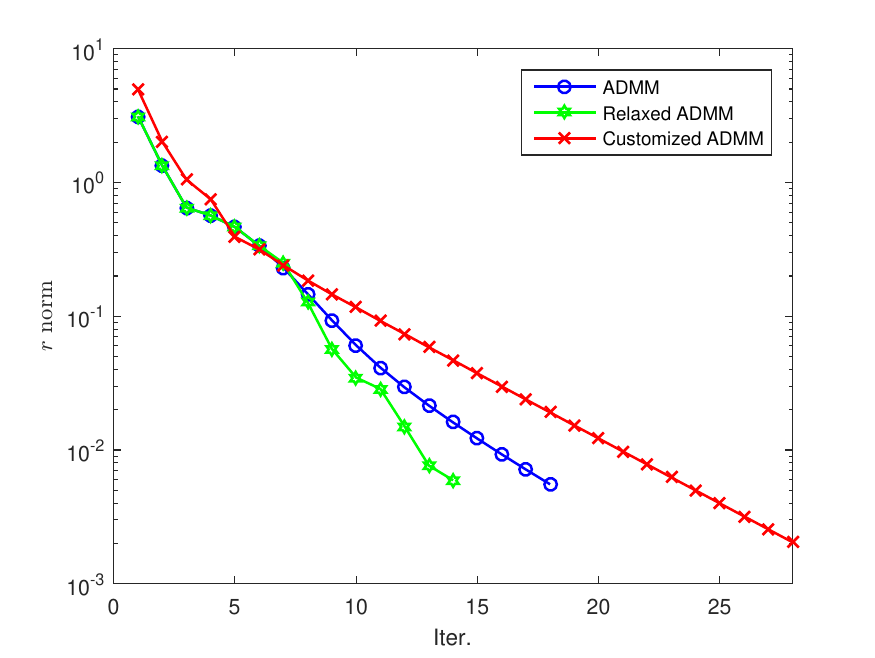}} \ 
  \scalebox{0.4}{\includegraphics{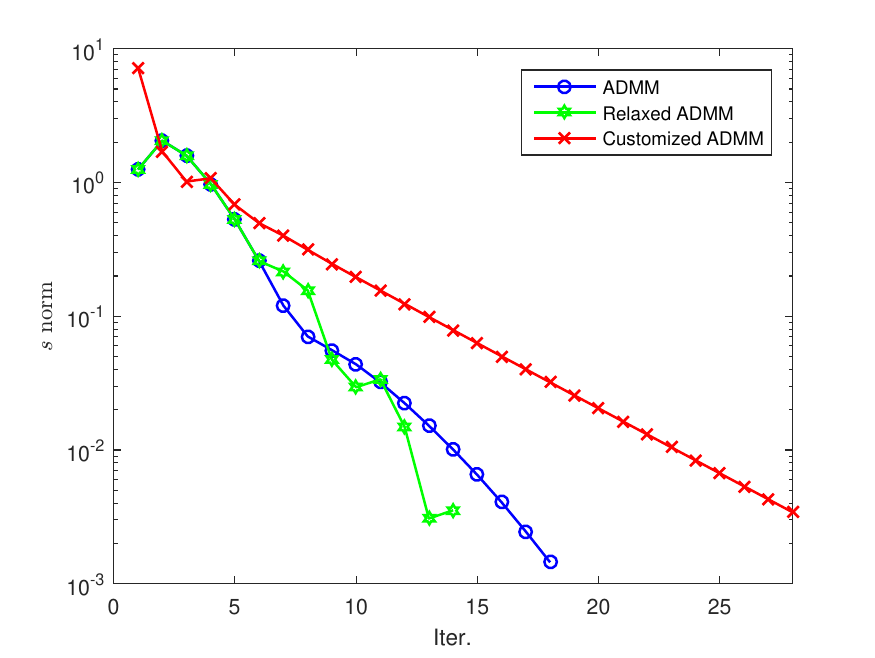}}
  }
   \caption{
   The primal and dual
residual for ADMM, relaxed ADMM and relaxed customized ADMM.}\label{fig1}
\end{center}
\end{figure*}

\subsection{Sparse inverse covariance selection}
In this section, we consider to use the ADMM to solve the sparse inverse covariance selection problem (SCSP), which first proposed in \cite{Arthur}. This problem can be understood as a
structure learning problem of estimating the topology of the undirected
graphical model representation of the Gaussian, and is a popular method using in reverse engineering of genetic regulatory networks. Suppose the vector $a_i,\ i=1,... N$  is a sample  from a zero mean Gaussian distribution
in $\Re^n$. i.e.,
\[
a_i \sim N(0,\Sigma), i= 1,...,N,
\]
but the  positive definite covariance matrix $\Sigma$ is unknown. Our  task  is to estimate the covariance matrix $\Sigma$ under the
assumption that $\Sigma^{-1}$ is sparse.
Let $S = (1/m)\sum_{i=1}^m a_i a_i^T$
be the empirical covariance matrix of the
sample, the convex formulation of SCSP is as follows.
\begin{equation}\label{covsel}
  \mbox{min} \{\mbox{\textbf{Tr}}(SX)-\log\det X +\tau\|X\|_1\ | \ X\in {\mathbb{S}}_{++}^n \}
\end{equation}
where  $\|\cdot\|_1  $ denotes the $l_1$ norm,  ${\mathbb{S}}_{++}^n$ denotes the set of symmetric positive definite $n\times n$ matrices, and $\mbox{\textbf{Tr}}(SX)$ is the trace of  $S X $. Now we show how to apply ADMM for SCSP. First, we rewrite the original problem (\ref{covsel}) into the canonical form in form of  (\ref{eq:1}) that consists of two sparable variables. Introducing a new variable $Y=X$, the problem can be written as
\begin{equation}\label{covsel:2}
    \begin{split}
\mbox{min}\quad & \mbox{\textbf{Tr}}(SX)-\log\det X +\tau\|Y\|_1\\
\mbox{s.t.}\quad &X-Y=0.
\end{split}
\end{equation}
Then using ADMM or \eqref{covsel:2}, we get
\begin{subequations}
\begin{numcases}{}
X^{k+1}={\rm arg}\!\min\{ \mbox{\textbf{Tr}}(SX)-\log\det X+\frac{\beta}{2}\|X-Y^k-\frac{1}{\beta}\Lambda^k\|_F^2\},\label{ADM-1}\\
Y^{k+1}={\rm arg}\!\min\{\tau\|Y\|_1+\frac{\beta}{2}\|X^{k+1}-Y-\frac{1}{\beta}\Lambda^k\|_F^2\},\label{ADM-2}\\
\Lambda^{k+1}=\Lambda^{k}-\beta\left(X^{k+1}-Y^{k+1}\right).\label{ADM-3}
\end{numcases}
\end{subequations}
The $X$-minimization can be solved by the  orthogonal eigenvalue
decomposition. The $X$-minimization can be solved  by the soft thresholding operator \eqref{soft-oper}. Thus the subproblems can all have closed form solutions. We refer the reader to e.g., \cite{Yuan2012} for more details.

The implementation details are as follows: we take the initial iterate as $Y^{0}= 0, \Lambda^{0}=0$. The relaxation factor is set to be $\gamma= 1.7$ and $\beta$ is all set to be $\beta =1$.the stopping criteria is set to be similar as  \eqref{error}. For different number of features $n$, we always set the number of samples  be $0.01*n^2$.
We also randomly generated ten cases for each instance, and the results reported   were averaged over ten runs. In Tables 4-6, we report the performance of the relaxed ADMM algorithm for different settings of the dimensions.
It shows that the relaxed ADMM is significantly faster than the ADMM and the relaxed customized admm.  For example, for $n=300$ with $\varepsilon^{\mbox{abs}}=10^{-6},\ \varepsilon^{\mbox{rel}}=10^{-4}$, ADMM and the relaxed customized ADMM solved it to the desired accuracy with $21$ and $29$ iterations respectively, while the relaxed ADMM took $14$ iterations. The numerical results of three methods for the case with $n=300$ are plotted in Fig 2. The left part in Fig 2 presents the relationship of primal residual $r^k$ and the number of iteration while the right part shows the relationship of dual residual $s^k$ and the number of iteration. The pink curve plots our relaxed ADMM and the other two curves belongs to ADMM and relaxed customized ADMM respectively. We can observe our method converges linearly and is considerably faster than other two methods, which supports our convergence analysis.

\begin{table*}[!ht]
\newcommand{\cell}[1]{#1}
\newcommand{\cellc}[1]{\textrm{\color{blue}#1}}
\begin{center}
\caption{Performance comparison of ADMM, Relaxed ADMM and Relaxed customized ADMM($\varepsilon^{\mbox{abs}}=10^{-4},\ \varepsilon^{\mbox{rel}}=10^{-2}$) }\label{pca4}
\vskip0.25cm
\begin{tabular}{|c|c|c|c|c|c|c|c|c|c|c|c|c|}\hline
\multicolumn{1}{|c|}{  matrix} & \multicolumn{4}{|c|}{ ADMM}& \multicolumn{4}{|c|}{ Relaxed ADMM} & \multicolumn{4}{|c|}{ Relaxed customized ADMM}  \\ \hline
 $n $    &  \cell{ Iter.  }&  \cell{$\|r^{k}\|_2$  }&  \cell{$\|s^{k}\|_2$   } &  \cell{Time }  & \cell{Iter.   } &  \cell{$\|r^{k}\|_2$   }&  \cell{$\|s^{k}\|_2$ }&  \cell{Time } &  \cell{ Iter.  }&  \cell{$\|r^{k}\|_2$  }&  \cell{$\|s^{k}\|_2$   } &  \cell{Time } \\ \hline
  200	 &  11	 &  1.25e-01 &  5.16e-03 &      0.20	 &    9	 &  1.09e-01 &  6.28e-03 &   0.15	 &   16	 &  1.00e-01 &  1.33e-01 &    0.28\\  \hline
  300	 &   9	 &  1.55e-01 &  9.21e-03 &      0.37	 &    7	 &  1.20e-01 &  7.77e-03 &    0.28 &   16	 &  1.25e-01 &  1.61e-01 &     0.64\\  \hline	500	 &   7	 &  2.06e-01 &  9.81e-02 &      0.88	 &    6	 &  2.06e-01 &  4.45e-02 &    0.78 &   16	 &  1.65e-01 &  2.04e-01 &    2.19	\\  \hline
  700	 &   6	 &  2.40e-01 &  3.28e-01 &      1.71	 &    6	 &  1.43e-01 &  1.37e-02 &     1.77  &   16	 &  1.98e-01 &  2.37e-01 &   5.09	\\  \hline
  900	 &  6	     &  2.53e-01 &  3.85e-01 &      3.26	 &    7	 &  8.36e-02 &  8.46e-03 &    4.11 &   15	 &  3.24e-01 &  3.83e-01 &   9.14	\\  \hline
   1100	 &  7	     &  1.20e-01 &  1.95e-01 &      7.14	 &    5	 &  1.98e-01 &  3.47e-02 &      5.15&   15	 &  3.61e-01 &  4.19e-01 &   16.84 \\ \hline
\end{tabular}
\end{center}
\vskip-0.4cm
\end{table*}

\begin{table*}[!ht]
\newcommand{\cell}[1]{#1}
\newcommand{\cellc}[1]{\textrm{\color{blue}#1}}
\begin{center}
\caption{Performance comparison of ADMM, Relaxed ADMM and Relaxed customized ADMM($\varepsilon^{\mbox{abs}}=10^{-5},\ \varepsilon^{\mbox{rel}}=10^{-3}$) }\label{pca5}
\vskip0.25cm
\begin{tabular}{|c|c|c|c|c|c|c|c|c|c|c|c|c|}\hline
\multicolumn{1}{|c|}{  matrix} & \multicolumn{4}{|c|}{ ADMM}& \multicolumn{4}{|c|}{ Relaxed ADMM} & \multicolumn{4}{|c|}{ Relaxed customized ADMM}   \\ \hline
 $n $    &  \cell{ Iter.  }&  \cell{$\|r^{k}\|_2$  }&  \cell{$\|s^{k}\|_2$   } &  \cell{Time }  & \cell{Iter.   } &  \cell{$\|r^{k}\|_2$   }&  \cell{$\|s^{k}\|_2$ }&  \cell{Time } &  \cell{ Iter.  }&  \cell{$\|r^{k}\|_2$  }&  \cell{$\|s^{k}\|_2$   } &  \cell{Time } \\ \hline
    200	 &   19	 &  1.08e-02 &  6.67e-05 &  0.35	 &   14	 &  1.01e-02 &  9.29e-05 &   0.25	 &   22	 &  1.18e-02 &  1.56e-02 &   0.40 \\ \hline
  300	 &   15	 &  1.49e-02 &  1.10e-04 &    0.63 &   11	 &  7.29e-03 &  3.31e-05 &   0.45	 &   22	 &  1.47e-02 &  1.90e-02 &   0.91 \\ \hline
  500	 &   12	 &  1.99e-02 &  1.05e-03 &    1.53	 &    9	 &  1.74e-02 &  1.76e-04 &    1.21	 &   22	 &  1.94e-02 &  2.40e-02 &   2.96	 \\ \hline
  700	 &   11	 &  1.92e-02 &  4.19e-03 &    3.43  &    8	 &  2.72e-02 &  3.43e-04 &   2.46 &   22	 &  2.34e-02 &  2.79e-02 &    7.05 \\ \hline
  900	 &   10	 &  3.20e-02 &  1.23e-02 &    5.85 &    9	 &  1.44e-02 &  1.22e-04 &    5.35	 &   22	 &  2.67e-02 &  3.15e-02 &     13.64\\  \hline
  1100	 &   10	 &  2.30e-02 &  1.59e-02 &    10.65&    8	 &  1.83e-02 &  2.92e-04 &    8.70 &   22 &  2.97e-02 &  3.44e-02 &   25.49 \\ \hline
\end{tabular}
\end{center}
\vskip-0.4cm
\end{table*}

\begin{table*}[!ht]
\newcommand{\cell}[1]{#1}
\newcommand{\cellc}[1]{\textrm{\color{blue}#1}}
\begin{center}
\caption{Performance comparison of ADMM, Relaxed ADMM and Relaxed customized ADMM ($\varepsilon^{\mbox{abs}}=10^{-6},\ \varepsilon^{\mbox{rel}}=10^{-4}$) }\label{pca6}
\vskip0.25cm
\begin{tabular}{|c|c|c|c|c|c|c|c|c|c|c|c|c|}\hline
\multicolumn{1}{|c|}{  matrix} & \multicolumn{4}{|c|}{ ADMM}& \multicolumn{4}{|c|}{ Relaxed ADMM}  & \multicolumn{4}{|c|}{ Relaxed customized ADMM} \\ \hline
 $n $    &  \cell{ Iter.  }&  \cell{$\|r^{k}\|_2$  }&  \cell{$\|s^{k}\|_2$   } &  \cell{Time }  & \cell{Iter.   } &  \cell{$\|r^{k}\|_2$   }&  \cell{$\|s^{k}\|_2$ }&  \cell{Time } &  \cell{ Iter.  }&  \cell{$\|r^{k}\|_2$  }&  \cell{$\|s^{k}\|_2$   } &  \cell{Time } \\ \hline
 200	 &   26	 &  1.29e-03 &  1.13e-06 &      0.51	 &   19	 &  9.84e-04 &  4.96e-07 &      0.36	 &   29	 &  9.73e-04 &  1.29e-03 &   0.56	 \\ \hline
  300&   21	 &  1.29e-03 &  1.11e-06 &      0.85	 &   14	 &  1.43e-03 &  1.96e-06 &      0.54	 &   29	 &  1.21e-03 &  1.56e-03 &    1.21	 \\ \hline
  500&  17	 &  1.74e-03 &  1.23e-05 &      2.31	 &   12	 &  1.29e-03 &  9.19e-07 &      1.63	 &   29	 &  1.60e-03 &  1.98e-03 &   3.88 \\ \hline
  700&  15	 &  2.32e-03 &  1.33e-04 &      4.64	 &   11	 &  1.49e-03 &  1.25e-06 &      3.58	 &   28	 &  2.75e-03 &  3.28e-03 &   8.98 \\ \hline
  900&  15	 &  2.13e-03 &  1.76e-04 &      8.90	 &   11	 &  2.42e-03 &  4.37e-06 &      6.64	 &   28	 &  3.14e-03 &  3.70e-03 &  17.40 \\ \hline
  1100&  14 &  2.43e-03 &  5.71e-04 &   15.61   &   10	 &  2.81e-03 &  6.40e-06 &     11.15	 &   28	 &  3.50e-03 &  4.05e-03  & 32.00 \\ \hline
\end{tabular}
\end{center}
\vskip-0.4cm
\end{table*}

\begin{figure*}[ht!]
         \begin{center}
  \centerline{
  \scalebox{0.4}{\includegraphics{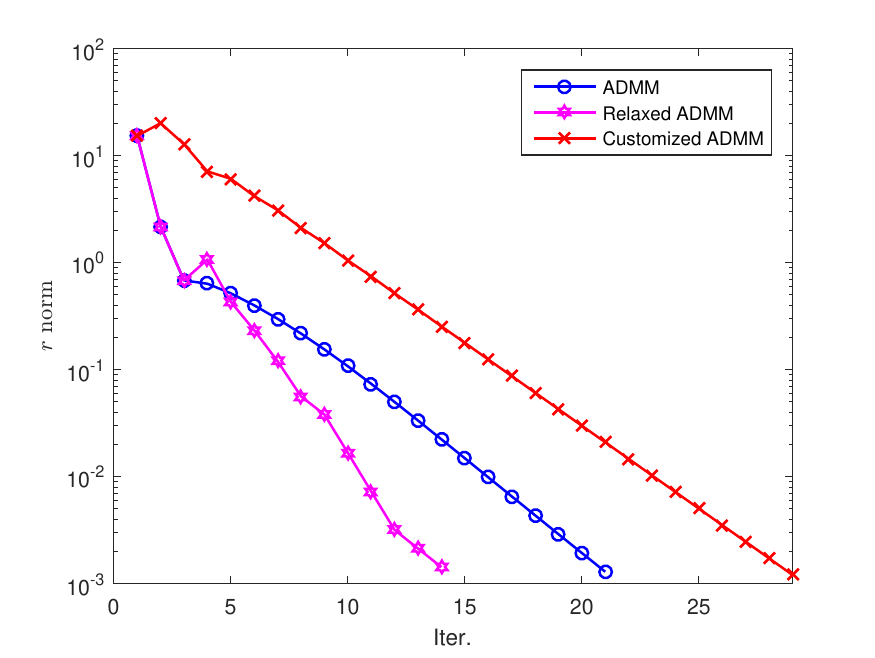}} \ 
  \scalebox{0.4}{\includegraphics{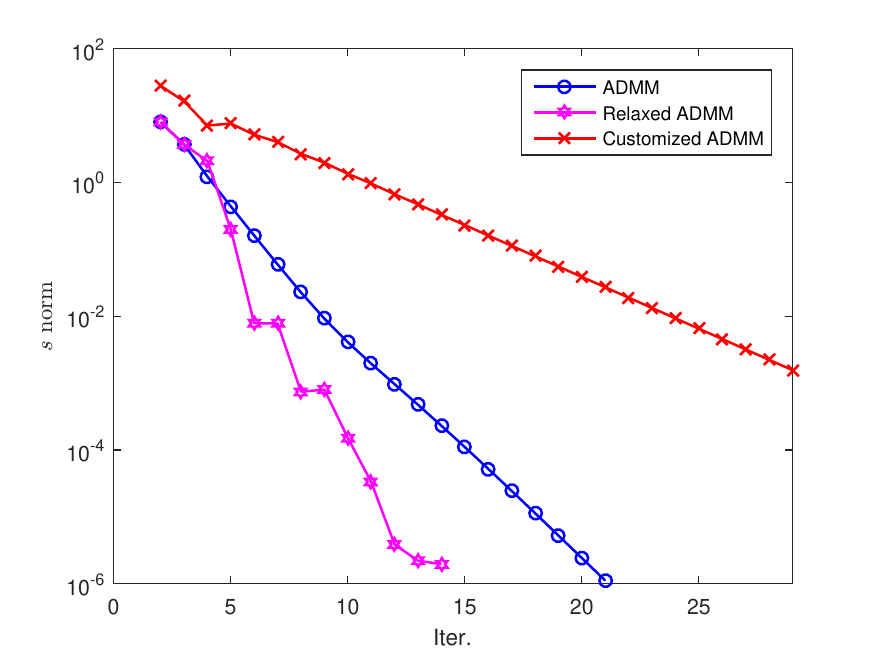}}
  }
   \caption{
   The primal and dual
residual for ADMM, relaxed ADMM and relaxed customized ADMM. ($n=300$)}\label{fig2}
\end{center}
\end{figure*}

\section{Conclusions}\label{section:conclusions}

In this paper, we propose a over-relaxed ADMM method and introduce a simple criterion. We show that, when the criterion is satisfied at each iteration, we can over-relax the variables of ADMM. We prove the global convergence of the over-relaxed ADMM. Finally, we demonstrate the numerical performance improvement of the this algorithm by solving Lasso and sparse inverse covariance selection problems. In the future, we will investigate the linear convergence rate of our algorithm under some regularity conditions for error bounds and extend our over-relaxation scheme to the multi-block ADMM.


\end{document}